\documentclass{article}
\usepackage{tikz}
\usetikzlibrary{arrows.meta}
\usepackage{amsmath, amssymb}
\usepackage{enumitem}
\usepackage{graphicx}
\usepackage{multirow}
\usepackage{amsmath,amssymb,amsfonts}
\usepackage{amsthm}
\usepackage{mathrsfs}
\usepackage[title]{appendix}
\usepackage{xcolor}
\usepackage{textcomp}
\usepackage{manyfoot}
\usepackage{booktabs}
\usepackage{algorithm}
\usepackage{algorithmicx}
\usepackage{algpseudocode}
\usepackage{listings}
\usepackage{dsfont}
\newcommand {\R}{\mathbb{R}}
\newcommand {\Z}{\mathbb{Z}}
\newcommand {\N}{\mathbb{N}}
\newcommand{\pP}{\mathbb{P}}
\newcommand{\bP}[1]{\mathbb{P}\left(#1\right)}
\newcommand{\E}{\mathbb{E}}
\newcommand{\bE}[1]{\mathbb{E}\left[#1\right]}
\newcommand{\CE}[2]{\mathbb{E}\left[#1\;|\;#2\right]}

\newcommand\abs[1]{\left|#1\right|}

\newcommand{\bb}[1]{\left(#1\right)}
\newcommand{\Cb}[1]{\left[#1\right]}
\newcommand{\cb}[1]{\left\{#1\right\}}

\newcommand{\ind}{\mathds{1}}

\newtheorem{theo}{Theorem}[section]

\newtheorem{lem}[theo]{Lemma}

\theoremstyle{definition}
\newtheorem{defi}[theo]{Definition}

\newtheorem{rem}[theo]{Remark}

\newtheorem*{conjecture}{Conjecture}

\begin{document}

\title{Ergodicity Criterion for One-Sided, One-Dimensional IPS with a Long-Lived State}
\author{Maciej Głuchowski, Georg Menz}
\date{}
\maketitle
\begin{abstract}
We study one-dimensional, one-sided, nearest-neighbor Interacting Particle Systems (IPS) with positive rates and identify a criterion for ergodicity based on the presence of a \emph{long-lived state} a site can occupy. The criterion is that the IPS admits a state with a ratio of rates of exiting to entering below~$\sqrt{2}$. The main idea is that under the canonical coupling, it is possible to identify large blocks in spacetime where two coupled trajectories must agree to share the special state, which inhibits the spread of disagreement. The result covers much of the parameter region left open by earlier work on the Positive Rates Conjecture (PRC) for simple IPS, narrowing the unresolved IPS to noisy versions of the East model.
\end{abstract}
\section{Introduction}

Interacting particle systems (IPS) provide a class of toy models that are simple enough to allow for rigorous mathematical analysis and yet are subtle enough to show a plethora of interesting physical phenomena, e.g.~phase transitions. The broader theme of this article is the question of whether a given IPS is ergodic, which means, in layman's terms, that it will forget its initial state over time.\\

To be more precise, an IPS is a Markov process on the space of configurations~$\mathcal{A}^{\Lambda}$
for some alphabet of possible states~$\mathcal{A}$ (usually~$\cb{0,\,1}$) and lattice~$\Lambda$ (usually~$\Z^d$) where updates occur independently at each site, triggered by exponential clocks. Once a clock rings at some~$j\in \Lambda$, the process has an opportunity to change the state at that site depending on the states of the site's neighbors~$\mathcal{N}_j$ (usually~$\mathcal{N}_j = j + \mathcal{N}$, with~$\mathcal{N}\Subset \Lambda$). These transitions are governed by transition rates~$ c_j (a,\,\zeta)$ or, equivalently, by the matrix of transition probabilities~$ P_j\bb {a\mid\zeta}$. The latter is the convention we use and is to be interpreted as the probability for site~$j$ to pick its new state as~$a\in\mathcal{A}$ given the trajectory is in configuration~$\zeta\in\mathcal{A}^{\Lambda}$ immediately before an update at~$j$. We identify the IPS with its transition matrix and write~IPS$(P)$ to refer to the IPS with matrix~$\cb{P_j\bb{a\mid\zeta}}_{j\in\Lambda,\,a\in\mathcal{A},\,\zeta\in\mathcal{A}^{\Lambda}}$. For a reader new to IPS, we point to an introductory paper by Durrett~\cite{Durrett:81}, and the textbook on IPS by Liggett~\cite{Liggett:05}. For precise definitions and the construction of the class of IPS studied in this article, we refer to the companion article~\cite{GluMen:25}.\\

The main idea behind the result is to exploit the presence of long-lived states. Suppose that for some~$a\in\mathcal{A}$ the IPS rule is such that if a site enters the state~$a$, it takes a long time for it to change, regardless of the states of interacting neighbors. Under the canonical coupling, we can then identify large chunks in spacetime, where all coupled trajectories must agree. These chunks block the flow of information, and with the assumptions of positive rates, they appear with some positive frequency. In the one-dimensional setting, their presence can be exploited to prove ergodicity via a random-walk argument. The main requirement of this method is that the maximal rate at which blocking chunks disappear is sufficiently low when compared to the minimal rate at which they appear. Direct calculation shows this ratio must be below~$\sqrt{2}$. The resulting criterion is easy to check and nevertheless quite effective.  \\

We apply the result to continue our previous work in Chapter~$7$ of~\cite{GluMen:25}, where we focus on the problem of proving that all~\emph{simple} IPS with positive rates are ergodic. An IPS is simple if it is nearest-neighbor, one-sided, one-dimensional, with two states. The problem is reduced to a small region in the four-dimensional parameter space, where no known proof methods apply. The main result of this work is an ergodicity criterion in the form of a simple inequality, which covers most of the remaining region. The only simple IPS with positive rates for which determination of ergodicity remains an open problem are IPS acting like a noisy version of the East model. \\

Random-walk-based coupling arguments are standard for determining properties of IPS or Markov processes in general. Examples include the celebrated proof of PRC for attractive, nearest-neighbor IPS due to Gray~\cite{Gray:82} or Durrett's study of drift of the supercritical contact process~\cite{Durrett:80:Contact}.

\section{Positive rates conjecture for simple IPS}

This article is a continuation of our work in Chapter~$7$ of~\cite{GluMen:25}, where we investigate the simplest non-trivial case of the celebrated Positive Rates Conjecture (PRC). To elaborate, we need to recall the definitions of~\emph{ergodicity} and~\emph{positive rates}.

\begin{defi}[Ergodicity]\label{def:ergodicity}
    An $\text{IPS}(P)$ is said to be \emph{ergodic} if it has an attractive distribution.
\end{defi}
\begin{defi}[Positive rates]\label{positive rates}
     A transition matrix $P$ has \emph{positive rates} if it differs from the identity on every entry.
\end{defi}
The positive rates property ensures that the IPS is locally irreducible, meaning that, upon update, a site can transition into any state with a positive probability. This introduces noise to the system, which in the one-dimensional, homogeneous ($P$ is translation invariant) setting is conjectured to be ergodic, i.e.~prevents the IPS from remembering its original configuration. This is the subject of the PRC.
\begin{conjecture}[Positive Rates Conjecture]
    Every IPS on a one-dimensional lattice with homogeneous interactions of bounded range and positive rates is ergodic.
\end{conjecture}
While there exists a counterexample to PRC published by G\'acs in~\cite{Gacs:86} which sports an alphabet of size~$2^{293}$, we believe that the result holds for sufficiently basic IPS. The most  straightforward class of IPS, for which the question of ergodicity is non-trivial, is termed \emph{simple}:
\begin{defi}(Simple, one-sided, nearest-neighbor IPS)
    An IPS is~\emph{simple} if it is one-dimensional, homogeneous with one-sided nearest-neighbor interactions, meaning~$\Lambda=\Z$,~$\mathcal{N}_j = \cb{j,\,j+1}$ and~$P_j\bb{\,\cdot\mid\zeta} = P_0\bb{\,\cdot\mid T_{-j}(\zeta)}$.
\end{defi}

\begin{conjecture}[PRC for simple IPS]\label{pro:PRC_simple}
    Every simple IPS with positive rates is ergodic.
\end{conjecture}
For the classification of simple IPS and examination of simulations that support the conjecture, we refer the reader to Chapter~$3$ of~\cite{GluMen:25}. Here, we briefly examine the current state of the conjecture systematically. A simple IPS is determined by four parameters
\begin{equation*}
    p_{1|11} := P(1\;|\;11),\quad p_{1|10}:= P(1\;|\;10),\quad p_{1|01}:= P(1\;|\;01),\quad p_{1|00}:= P(1\;|\;00).
\end{equation*}
The positive rates property is equivalent to
\begin{equation*}
    p_{1|11}<1,\quad p_{1|10}<1,\quad p_{1|01}>0,\quad p_{1|00}>0.
\end{equation*}
A consequence of the Time-Scaling Lemma (Lemma~$4.1$ in~\cite{GluMen:25}) is that for~$\lambda>0$ the ergodicity of a simple IPS$(P)$ is equivalent to ergodicity of the IPS with parameters~$P^{\lambda} = \bb{1-\lambda}I + \lambda P$ i.e.
\begin{align*}
    P^{\lambda}\bb{1\mid 11}&=1-\lambda + \lambda p_{1\mid 11},\; &P^{\lambda}\bb{1\mid 10}&=1-\lambda + \lambda p_{1\mid 10},\\\;P^{\lambda}\bb{1\mid 01}&=\lambda p_{1\mid 01},\;&P^{\lambda}\bb{1\mid 00}&=\lambda p_{1\mid 00}.
\end{align*}
For any IPS($P$) consider the line in the parameter space~$[0,1]^4$ passing through~$P$ and~$I$. All IPS on this line define the same process up to time-scaling, with IPS closer to~$I$ being slowed down. For every IPS there exists it's least slowed down version, whose parameters lie on one of the three-dimensional faces of~$[0,1]^4$, located on the opposite side from~$I$.
The question of simple PRC hence reduces to the verification of the conjecture for parameters on those faces, i.e. that lie in the union of four cubes
\begin{equation*}
    \cb{p_{1\mid 11} = 0}\cup \cb{p_{1\mid 00} = 1}\cup
    \cb{p_{1\mid 10} = 0}\cup\cb{p_{1\mid 01} = 1}.
\end{equation*}
 Further simplification is possible since switching the roles of states~$0$ and~$1$ does not affect the dynamics of the IPS. Modulo this transformation the pairs of cubes~$\cb{p_{1\mid 11} = 0},\,\cb{p_{1\mid 00} = 1}$ and~$\cb{p_{1\mid 10} = 0},\,\cb{p_{1\mid 01} = 1}$ hold the same IPS. Thus, the simple PRC is reduced to the cases
 \begin{equation*}
      \cb{p_{1\mid 11} = 0}\cup \cb{p_{1\mid 10} = 0}.
 \end{equation*}
In Chapter~$7$ of~\cite{GluMen:25} we show that PRC holds for IPS with parameters in the cube~$\cb{p_{1\mid 10} = 0}$ and for the region of the~$\cb{p_{1\mid 11 = 0}}$ cube that intersects with
\begin{equation*}
    \cb{p_{1|10}<\frac{1}{2}, p_{1|01}>0, p_{1|00}>0}\cup\cb{p_{1|10}< p_{1|01}+p_{1|00}}\cup\cb{p_{1\mid 01} > p_{1\mid 00}}.
\end{equation*}
For an illustration of the remaining region we refer to Figure~\ref{fig:solved_cases}
\begin{figure}[H]\label{fig:solved_cases}

\begin{tikzpicture}[scale=4, line cap=round, line join=round]
  \begin{scope}
    \draw[->] (-0.02,0) -- (1.08,0) node[below right] {$p_{1\mid 10}$};
    \draw[->] (0,-0.02) -- (0,1.08) node[above left] {$p_{1\mid 01}$};
    \draw (0,0) rectangle (1,1);

    \foreach \t in {0,1} {
      \draw (\t,0) -- (\t,-0.02) node[below] {\small $\t$};
      \draw (0,\t) -- (-0.02,\t) node[left]  {\small $\t$};
    }
    \draw (0.5,0) -- (0.5,-0.02) node[below] {\small $\tfrac12$};
    \draw (0,0.5) -- (-0.02,0.5) node[left]  {\small $\tfrac12$};

    \fill[gray!25] (0,0) -- (0.5,0) -- (0.5,1) -- (0,1) -- cycle;
    \fill[gray!25] (0.5,0.5) -- (1,1) -- (0.5,1) -- cycle;

    \draw[thick] (0,0) -- (0.5,0.5);
    \draw[thick,dashed] (0.5,0.5) -- (1,1);
    \draw[thick,dashed] (0.5,0) -- (0.5,0.5);
    \draw[thick] (0.5,0.5) -- (0.5,1);

    \node at (0.5,-0.18) {$p_{1\mid 00}=0$};
  \end{scope}

  \begin{scope}[xshift=1.4cm] 
    \draw[->] (-0.02,0) -- (1.08,0) node[below right] {$p_{1\mid 10}$};
    \draw[->] (0,-0.02) -- (0,1.08) node[above left] {$p_{1\mid 01}$};
    \draw (0,0) rectangle (1,1);

    \foreach \t in {0,1} {
      \draw (\t,0) -- (\t,-0.02) node[below] {\small $\t$};
      \draw (0,\t) -- (-0.02,\t) node[left]  {\small $\t$};
    }
    \draw (0.5,0) -- (0.5,-0.02) node[below] {\small $\tfrac12$};
    \draw (0,0.25) -- (-0.02,0.25) node[left] {\small $\tfrac14$};
    \draw (0,0.75) -- (-0.02,0.75) node[left] {\small $\tfrac34$};

    \fill[gray!25] (0,0) -- (0.5,0) -- (0.5,1) -- (0,1) -- cycle;
    \fill[gray!25] (0.5,0.25) -- (1,0.75) -- (1,1) -- (0.5,1) -- cycle;
    \fill[gray!25] (0,0) -- (1,0) -- (1,0.25) -- (0,0.25) -- cycle;
    
    \draw[thick] (0.25,0) -- (0.5,0.25); 
    \draw[thick,dashed] (0.5,0.25) -- (1,0.75);
    \draw[thick] (0.5,0) -- (0.5,0.25);
    \draw[thick] (0,0.25) -- (0.5,0.25);
    \draw[thick,dashed] (0.5,0.25) -- (1,0.25);
    \node at (0.5,-0.18) {$p_{1\mid 00}=\frac{1}{4}$};
  \end{scope}
\end{tikzpicture}
         \caption{Sections of the cube~$\cb{p_{1\mid 11} = 0}$ with~$p_{1\mid 00} \in \left\{0, 1/4\right\}$.  The parameter~$p_{1\mid 10}$ varies on the~$X$ axis and the parameter~$p_{1\mid 01}$ on the~$Y$ axis. The black region represents the IPS for which ergodicity has been proven. For~$p_{1\mid 00}>\frac{1}{2}$ the entire section is in the covered region.}
\end{figure}
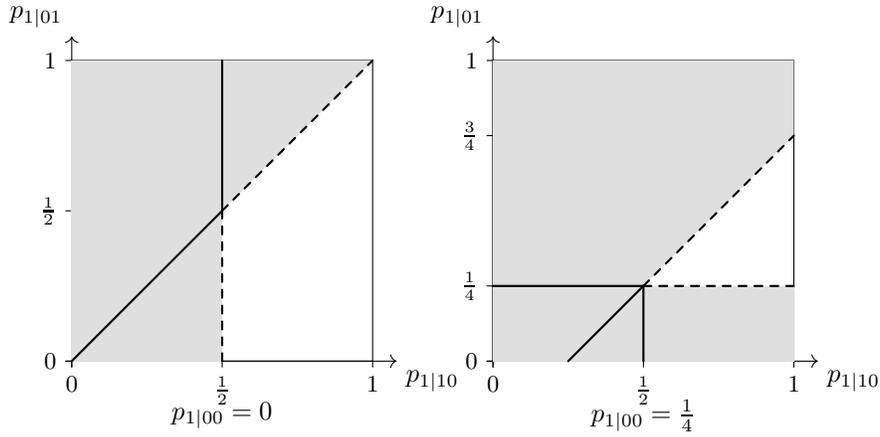

\section{Main result and discussion}\label{sec:main_result}

 In the main result of this work, we show that if there is a state~$a \in \mathcal{A}$ that is born sufficiently more frequently than it dies, then the IPS is ergodic.

\begin{theo}[Main result]\label{pro:main_result} 
    For a homogeneous, one-dimensional, one-sided, nearest-neighbor IPS$\bb{P}$ and a fixed~$a\in\mathcal{A}$ define coefficients
    \begin{align*}
         \beta &:= \min_{\zeta} P_0(a\mid\zeta),\\
         \delta&:= \max_{\zeta:\zeta(0)=a} 1-P_0(a\mid \zeta).
    \end{align*}
    If~$\delta < \sqrt{2}\beta$ then the IPS is ergodic.
\end{theo}
Heuristically, the statement of the last theorem is plausible as the state~$a$ dominates the trajectories of the IPS, which inhibits the spread of information between distant sites. For an illustration, we refer to Figure~\ref{fig:wall_IPS}. 

\begin{figure}[H]
        \centering
        \includegraphics[scale=0.2]{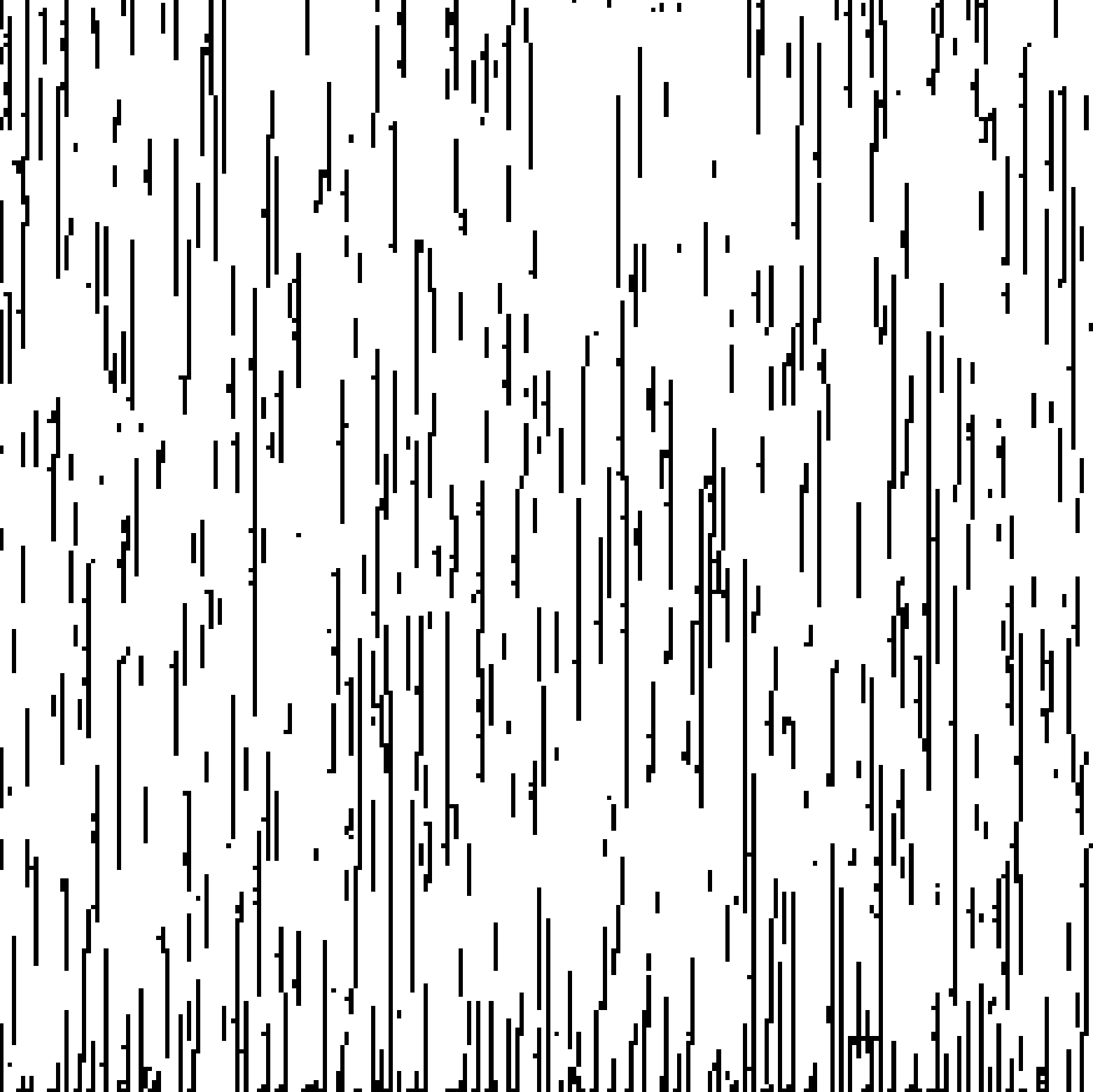}
        \caption{Simulated trajectory of the IPS with parameters~\mbox{$(p_{1|11},\,p_{1|10},\,p_{1|01},\,p_{1|00})=(0,\,0.9,\,0.02,\,0.02)$}. Black represents state~$1$, white state~$0$. The arrow of time points up. The initial configuration is random, sampled uniformly. We observe that while~$1$s annihilate each other, the state~$0$ is almost stationary, building \emph{walls} blocking the flow of information. Because of this typical shape of the trajectories those IPS were called \emph{wall IPS} in the companion article~\cite{GluMen:25}.}\label{fig:wall_IPS}
\end{figure}

\begin{figure}[H]
        \centering
        \begin{tikzpicture}[scale=4, line cap=round, line join=round]
  \begin{scope}
    \draw[->] (-0.02,0) -- (1.08,0) node[below right] {$p_{1\mid 10}$};
    \draw[->] (0,-0.02) -- (0,1.08) node[above left] {$p_{1\mid 01}$};
    \draw (0,0) rectangle (1,1);

    \foreach \t in {0,1} {
      \draw (\t,0) -- (\t,-0.02) node[below] {\small $\t$};
      \draw (0,\t) -- (-0.02,\t) node[left]  {\small $\t$};
    }
    \draw (0.5,0) -- (0.5,-0.02) node[below] {\small $\tfrac12$};
    \draw (0,0.5) -- (-0.02,0.5) node[left]  {\small $\tfrac12$};

    \fill[gray!25] (0,0) -- (0.5,0) -- (0.5,1) -- (0,1) -- cycle;
    \fill[gray!25] (0.5,0.5) -- (1,1) -- (0.5,1) -- cycle;

    \fill[gray!25, domain=0:1, variable=\x]
         plot ({\x}, {min(1,{1.4142*(1-\x)})}) -- (1,0) -- (0,0) -- cycle;
    \fill[red!25] (0.5,0)--(1,0)--(0.585,0.585)--(0.5,0.5)--cycle;
    \node at (0.5,-0.18) {$p_{1\mid 00}=0$};

    \draw[thick] (0.5,0)--(0.5,0.5);
    \draw[thick] (0.5,0.5)--(0.585,0.585);
    \draw[thick,dashed] (0.585,0.585)--(1,1);
    \draw[thick,dashed] (1,0)--(0.585,0.585);
    \draw[thick](0.585,0.585)--(0.293,1);
  \end{scope}

  \begin{scope}[xshift=1.4cm]
    \draw[->] (-0.02,0) -- (1.08,0) node[below right] {$p_{1\mid 10}$};
    \draw[->] (0,-0.02) -- (0,1.08) node[above left] {$p_{1\mid 01}$};
    \draw (0,0) rectangle (1,1);

    \foreach \t in {0,1} {
      \draw (\t,0) -- (\t,-0.02) node[below] {\small $\t$};
      \draw (0,\t) -- (-0.02,\t) node[left]  {\small $\t$};
    }
    \draw (0.5,0) -- (0.5,-0.02) node[below] {\small $\tfrac12$};
    \draw (0,0.25) -- (-0.02,0.25) node[left] {\small $\tfrac14$};
    \draw (0,0.75) -- (-0.02,0.75) node[left] {\small $\tfrac34$};

    \fill[gray!25] (0,0) -- (0.5,0) -- (0.5,1) -- (0,1) -- cycle;
    \fill[gray!25] (0.5,0.25) -- (1,0.75) -- (1,1) -- (0.5,1) -- cycle;
    \fill[gray!25] (0,0) -- (1,0) -- (1,0.25) -- (0,0.25) -- cycle;

    \fill[gray!25, domain=0:1, variable=\x]
         plot ({\x}, {min(1,{1.4142*(1-\x)})}) -- (1,0) -- (0,0) -- cycle;
    \fill[red!25] (0.5,0.25)--(0.823,0.25)--(0.689,0.439)--cycle;
    \draw[thick] (0.823,0)--(0.823,0.25);
    \draw[thick] (0.689,0.439) -- (0.293,1);
    \draw[thick,dashed] (1,0.25)--(0.823,0.25);
    \draw[thick,dashed] (0.823,0.25)--(0.689,0.439);
    \draw[thick,dashed] (0.689,0.439)--(1,0.75);
    \draw[thick](0.5,0.25)--(0.689,0.439);
    \draw[thick](0.5,0.25)--(0.823,0.25);
    \node at (0.5,-0.18) {$p_{1\mid 00}=\frac{1}{4}$};
  \end{scope}
\end{tikzpicture}

        \caption{Figure~$1$ updated by the results in Theorem~\ref{pro:main_result}. The red marks the new parameters corresponding to the IPS to which Theorem~\ref{pro:main_result} applies. The line~$\cb{1}\times[0,1]$ in the case~$p_{1\mid 00}=0$ corresponds to the  East model, after an application of the time-scaling lemma.}\label{fig:updated_region}
\end{figure}
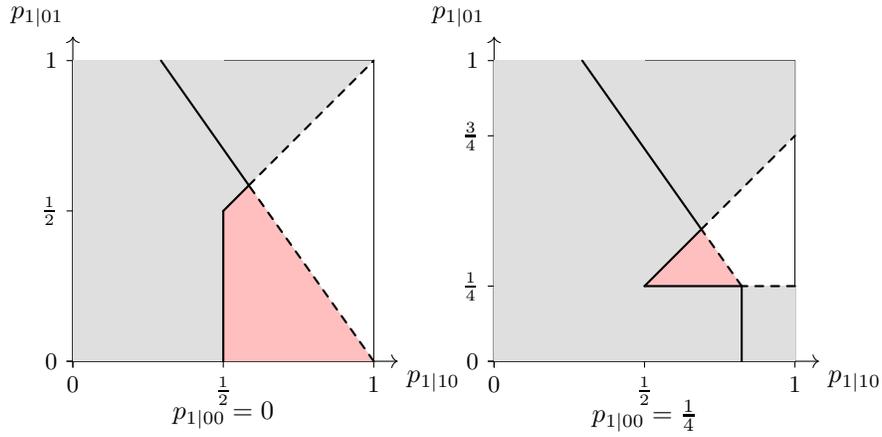

Due to Theorem~\ref{pro:main_result} we can update the updated region of simple IPS for which the PRC holds (see Figure~\ref{fig:updated_region}). The region for which ergodicity has not yet been proven lies next to the boundary
\begin{equation*}
    \cb{p_{1\mid 11} = 0,\,p_{1\mid 10} = 1,\,p_{1\mid 01}>0,\, p_{1\mid 00} = 0},
\end{equation*}
which corresponds to after an application of the time-scaling lemma to the set
\begin{align*}
    \cb{p_{1\mid 11} = p_{1\mid 01}>0,\,p_{1\mid 10} = 1,\, p_{1\mid 00} = 0}.
\end{align*}
The last set corresponds to the one-dimensional East model which is the classic example of a Kinetically Constrained Spin Model. The one-dimensional East model is not ergodic, having countably many extremal invariant measures, e.g.~$\delta_{\mathbf{0}}$ and~$\pi  = \otimes_{\Z}\text{Bern}(q)$. This is not a problem for the PRC for simple IPS as the East model does not have positive rates. However, the East model is almost ergodic: The~$\delta_{\mathbf{0}}$ measure is unstable, in the sense that the process starting from any configuration with infinite ones to the left converges to~$\pi$ (Theorem 3.1 in~\cite{CMST:10}). The IPS also has a positive spectral gap w.r.t.~$\pi$. For proof see Theorem~$4.2$ in~\cite{CMRT:06} (note switched roles of~$0$ and~$1$). Due to these properties, the addition of any noise immediately should make the process ergodic. However, a rigorous proof is missing as the addition of almost any type of noise makes the noisy model non-reversible. This makes it difficult to apply methods developed for the East model, as they heavily rely on the reversibility of the process and the fact that the invariant measure is explicitly known.\\

The proof, as described in the introduction, is a coupling argument that leverages the presence of the long lived state~$a\in\mathcal{A}$, which is unlikely to flip under any configuration of the neighbors. The positive rates assumption guarantees that with some frequency the symbol~$a$ spontaneously appears in all coupled trajectories. This agreement persists for a long time, blocking  the spread of disagreements between trajectories. Hence if the ratio between rates of births and deaths of the symbol~$a$ is sufficiently high then disagreements die out.\\

The method does not require the assumptions of one-sided, nearest-neighbor interactions. They can be dropped in trade-off for a weaker result. In particular, the assumption of one-sided interactions is only used for convenience. The restriction to one-dimensional IPS seems to be a hard limit, as in higher dimensions a disagreement can easily go around a blocked site. 

\section{Proof of main result}

The proof is via a coupling argument. Recall that~$\mathcal{A} = \cb{a_1,\,\dots,\,a_n}$ is a finite set called an alphabet, the elements of which we call states.  For any initial configuration $\zeta_0 \in \mathcal{A}^{\mathbb{Z}}$, the trajectory of the IPS is denoted by $\left\{ \zeta_t \right\}_{t \geq 0}$. We consider the canonical coupling: if we let $ (\tau_n^j)_{n \geq 0}^{j \in \mathbb{Z}} $ be the stopping times of the $ n $-th update at site $ j \in \mathbb{Z} $, and let $ (U_n^j)_{n \geq 0}^{j \in \mathbb{Z}} $ be i.i.d.~uniformly distributed random variables on~$[0,1]$. Then we set

\begin{align}\label{eq:coupling_construction}
\zeta_{\tau_n^j}(j) &:=  a_k\;\text{ if }\;  U_{n}^j \in \left[\sum_{m=1}^{k-1} P_j\bb{a_m\,|\,\zeta_{\tau_{n}^{j-}}},\; \sum_{m=1}^{k} P_j\bb{a_m\,|\,\zeta_{\tau_{n}^{j-}}}\right)\notag \\
\zeta_t(j) &:= \zeta_{\tau_n^j} \quad \text{for } t \in [\tau_n^j, \tau_{n+1}^j).
\end{align}
The probability of the process~$\zeta_t$ changing its state at site~$j$ to some~$a_k$ after an update at that site is~$P_j\bb{a_k\,|\,\zeta_{t^-}}$, when the left limit of the trajectory is the configuration~$\zeta_{t^-}$. This construction defines the trajectories of all initial configuration on a shared probability space, i.e.~the same choice of~$ (\tau_n^j)_{n \geq 0}^{j \in \mathbb{Z}} $~and $ (U_n^j)_{n \geq 0}^{j \in \mathbb{Z}} $. In order to study the ergodicity of the IPS we introduce the family of stopping times~$\bb{\pi_j}_{j\in\Z}$ given by
\begin{equation*}
\pi_j := \sup_{\zeta_0=\xi_0 \text{ off }j} \inf \{ t \geq 0 : \zeta_t = \xi_t \}.
\end{equation*}
The stopping time~$\pi_j$ is the first time after which  for all pairs of initial configurations that only differed at site~$j$ their trajectories converged. Hence, the system forgot the initial state at the site~$j$. The following lemma makes the connection between this notion and ergodicity.

\begin{lem}\label{pro:fast_agreement_implies_ergodicity}
    If an IPS has a finite range of interactions and 
    \begin{equation*}
         \lim\limits_{t \to \infty} t \mathbb{P}(\pi_0 > t) = 0,
    \end{equation*}
    then the IPS is ergodic.  
\end{lem}
\begin{rem}
    This observation is not unique to dimension one. In general if the lattice is~$\Z^d$ and~$\lim\limits_{t \to \infty}t^d \bP{\pi_0>t} = 0$ then ergodicity follows.
\end{rem}

Before we prove Lemma~\ref{pro:fast_agreement_implies_ergodicity} we need to recall the definition of~\emph{cone of dependence}. Intuitively, for a point in spacetime~$\bb{t,\,j}\in \R_+\times\Lambda$ the~$\mathbf{Cone}(t,\,j)$ is the set of all points in the past of~$t$ that could have communicated with the point~$(t,\,j)$. For formal construction, we point to Section~$2.1$ in~\cite{GluMen:25}.
\begin{figure}[H]
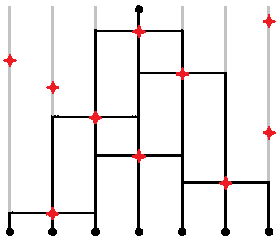

        \centering


        \caption{The cone of dependence with~$\Lambda = \Z$ and~$\mathcal{N}_j = \cb{j-1,\,j,\,j+1}$. The arrow of
time points up. The red markers indicate the update points.}
    \end{figure}
\begin{proof}[Proof of Lemma~\ref{pro:fast_agreement_implies_ergodicity}:]  
The canonical coupling does not allow for spontaneous birth of disagreements, meaning that for~$\zeta_t(j) \neq \xi_t(j)$ there must be a previous time~$s<t$ and~$i\in\mathcal{N}_j$ such that~$\zeta_s(i)\neq \xi_s(i)$. Thus, for canonically coupled $ \zeta_t $ and $ \xi_t$ to disagree at some site $ j \in \mathbb{Z} $, there must exist a path joining $ (t,j) $ with $ (0,i) \in \R_+\times \mathbb{Z} $, along which $ \zeta_t$ and $ \zeta_t$ disagree. This path must also be contained in $ \mathbf{Cone}(t,j) $.  Hence, if~$r\in\N$ is the interaction range then we have  

\begin{align*}
\mathbb{P}(\zeta_t(j) \neq \xi_t(j)) \leq& \mathbb{P}(\pi_i > t \text{ for some } i \in \mathbf{Cone}(t,j))\\
\leq& \mathbb{P}(|\mathbf{Cone}(t,j)| > 4rt)\\ 
&+ \mathbb{P}(\pi_i > t \text{ for some } i \in \{j-2rt, \dots, j + 4rt\})\\
\leq& e^{-2rt(3 \ln 3 -3)} + 4rt \mathbb{P}(\pi_0 > t) \xrightarrow[t \to \infty]{} 0.
\end{align*}

The last inequality holds as $ |\mathbf{Cone}(t,j)| $ is stochastically dominated by a Poisson point process with rate~$2r$, so we can bound the tail of its distribution by a Chernoff bound.
\end{proof}

To prove Theorem~\ref{pro:main_result} we use a random walk argument. The idea is to construct a pair of random walks that act as boundaries for disagreements, i.e.~the region in between them should contain all the points~$\bb{t,\,j}$ such that there exist~$\zeta_0 = \xi_0$ off~$0$ with~$\zeta_t(j)\neq \xi_t(j)$. If the RWs cross, then disagreements die out. Assuming a good control on how fast that occurs, Lemma~\ref{pro:fast_agreement_implies_ergodicity} can be employed. The following lemma makes this idea precise.

\begin{lem}\label{pro:random_walk_lemma}
    Suppose for a one-dimensional, one-sided IPS there exist random walks~$X_n$ and $Y_n$ on~$\R_{+}$ with the following properties:
    \begin{enumerate}
        \item The RW~$Y_n$ traces the spread of information from~$0$:
        \begin{equation*}
            Y_n\geq \inf\cb{t\in\R_+\;:\;0\in \mathbf{Cone}(t,\,-n)}.
        \end{equation*}
        \item The random walks bound the region with  (cf.~Figure~\ref{fig:hypo_walk}:
        \begin{equation*}
            \bigcup_{\zeta_0 = \xi_0 \text{ off }0}\cb{\bb{t,\,j}\in\R_+\times \Z\;:\;\zeta_t(j)\neq \xi_t(j)} \subseteq \bigcup_{n\geq 0} \left(Y_n,\, X_n \right] \times \cb{-n}.
        \end{equation*}
        \item For some step size~$r$ the walk~$Z_n := X_{rn}-Y_{rn}$ is eventually a supermartingale with negative drift, i.e for some~$\varepsilon>0$ there exists a~$T>0$ such that
        \begin{equation*}
            \CE{Z_{n+1} - Z_{n}}{\mathcal{F}_{rn}}\leq -\varepsilon
        \end{equation*}
        on the event~$\cb{X_{rn},\,Y_{rn}\geq T}$. Here~$\mathcal{F}_n$ is the canonical filtration for the process~$\bb{X_n,\,Y_n}$.
        \item The increments of~$X_n$ and~$Y_n$ are controlled: There are constants~$A,\,\alpha \geq 1$ such that 
        \begin{align*}
            \CE{\,\abs{X_{n+1}-X_n}^k}{\mathcal{F}_n} &\leq A k! \alpha^k,\\
            \CE{\,\abs{Y_{n+1}-Y_n}^k}{\mathcal{F}_n} &\leq A k! \alpha^k,\text{ for all } k\geq 0.\\
        \end{align*}
        \item There exists~$\lambda >0$ such that~$\bE{e^{\lambda X_0}}<\infty$. 
        \end{enumerate}
        Then the IPS is ergodic.
\end{lem}
In the proof of Lemma~\ref{pro:random_walk_lemma} we take advantage of the fact that the flow of information is directional - there is no communication from left to right. If to the right of site~$j$ there are no disagreements then the site will stay in agreement forever. The RW~$X_n$ traces the left-most sites for which we can guarantee this property and gives an upper bound for the region with potential disagreements. The lower bound, the process~$Y_n$ traces how fast the unique disagreement in the initial configuration spreads to the left.

\begin{figure}[H]
        \centering
        \includegraphics[scale = 0.33]{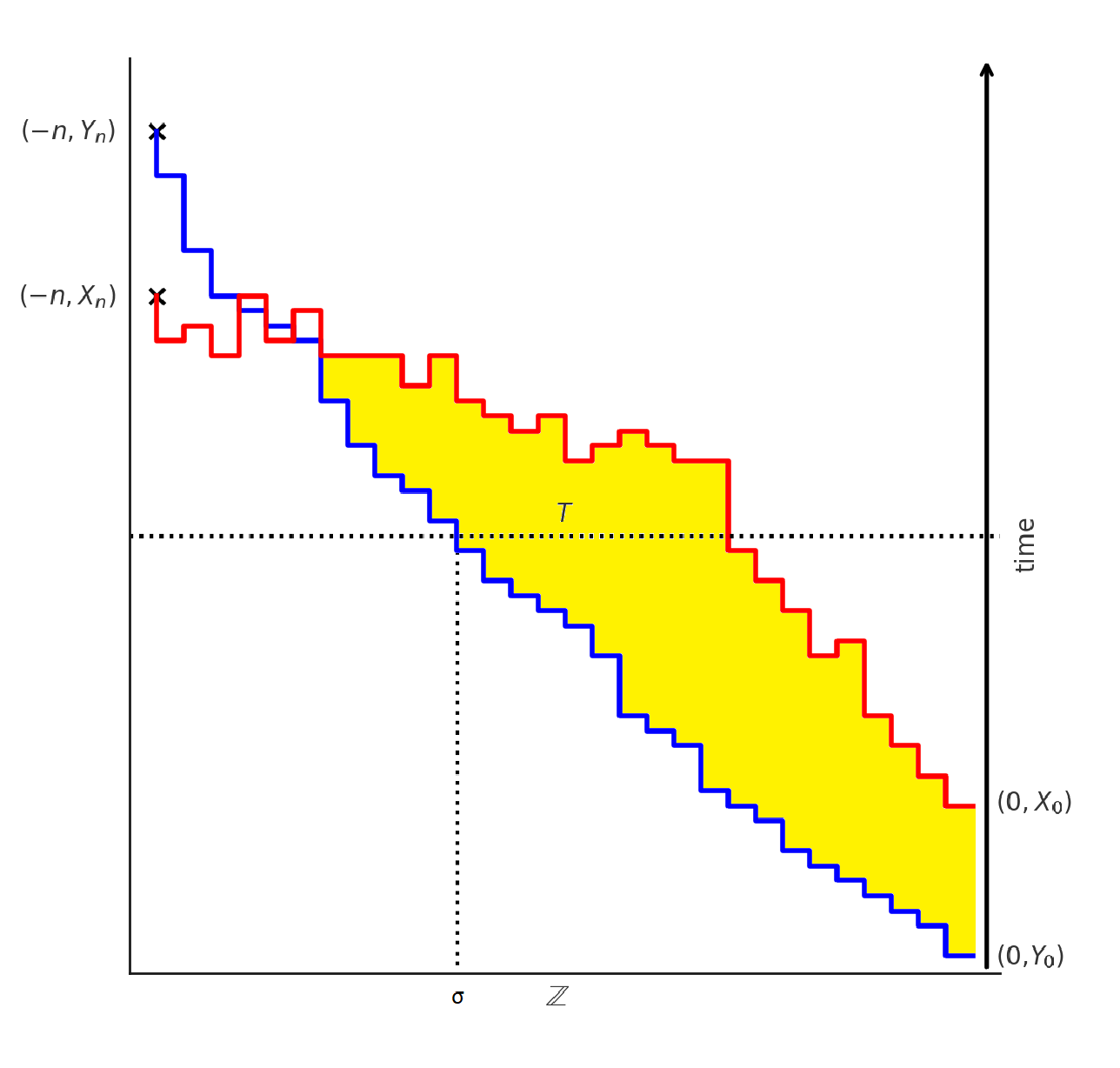}

        \caption{Representation of the random walks with the red and blue lines tracing points~$\bb{-n,\,X_n}$ and~$\bb{-n, \,Y_n}$ respectively. Initially the RWs may diverge but after they pass threshold~$T$ they tend to cross. The yellow region marks the points in spacetime where disagreements may occur, with the initial disagreement only present at site zero.}\label{fig:hypo_walk}
    \end{figure}

\begin{proof}[Proof of Lemma~\ref{pro:random_walk_lemma}]
Our aim to apply Lemma~\ref{pro:fast_agreement_implies_ergodicity}. To find suitable bounds on the tail of distribution of~$\pi_0$ we define the random variables
\begin{align*}
    \tau &:= \inf\cb{n\in \N\;:\; X_n \leq Y_n}= \inf\cb{n\in\N\;:\; Z_n \leq 0},\\
    M&:= \sup_{n\leq \tau} X_n.
\end{align*}
The variable~$M$ is the maximal time that the region between~$X_n$ and~$Y_n$ reaches, so~$\pi_0\leq M$. In particular if~$\bE{M}<\infty$ then~$t\bP{\pi_0>t}\to 0$, which proves the claim.  Our goal is now to show~$\bE{M}<\infty$. The following estimate shows a good control over the tail of~$\tau$ is sufficient:
\begin{align}\label{eq:M_estimate}
    \E\Cb{M} &\leq \bE{X_0}+\E\Cb{\sum_{n=0}^{\tau} \abs{X_{n+1}-X_n}}\notag\\
    &=\bE{X_0} +\sum_{n=0}^{\infty} \E\Cb{\abs{X_{n+1}-X_n}\ind\bb{\tau\geq n}}\notag\\
    &\leq\bE{X_0}+ \sum_{n=0}^{\infty} \E\Cb{\abs{X_{n+1}-X_n}^2}^{\frac{1}{2}}\pP\bb{\tau\geq n}^{\frac{1}{2}}\notag\\
    &\leq \bE{X_0} + (2A)^{\frac{1}{2}}\alpha \sum_{n=0}^{\infty} \pP\bb{\tau\geq n}^{\frac{1}{2}}.
\end{align}
Here we used Hölder's inequality and the assumption~$(4)$. By the assumption~$(5)$ the~$\bE{e^{\lambda X_0}}$ is finite hence so is~$\bE{X_0}$. To conclude that~$\bE{M}<\infty$ we just need a good control on the tail of~$\tau$.
We show that 
\begin{equation*}
     \pP\bb{n<\tau} =  O\bb{e^{-\gamma n}},
 \end{equation*}
for some~$\gamma>0$. To estimate the tail of~$\tau$ precisely we introduce one more stopping time:
\begin{equation*}
    \sigma := \inf\cb{n\in\N\;:\; 0\notin\mathbf{Cone}\bb{T,\,-n}}.
\end{equation*}
The variable~$-\sigma$ is the right-most negative site that had no chance to communicate with the site~$0$ up to time~$T$. By assumption~$(1)$ it is guaranteed that~$Y_{n}\geq T$ for~$n\geq \sigma$. Consequently if~$X_n \leq T$ for some~$n\geq \sigma$ then~$X_n\leq Y_n$ and~$\tau \leq n$. Because we only know that the walk~$Z_n$ becomes a supermartingale~\emph{eventually} we define its shifted modification~$S_n := Z_{\sigma + n}$. To simplify notation we also introduce~$\rho := \min\cb{0,\,\tau-\sigma}$, i.e. the stopping time of~$S_n$ passing below zero and~$\mathcal{G}_n$, the~$\sigma$-algebras generated by the RW~$S_n$ and~$\mathcal{F}_{\sigma}$. By assumption~$(3)$ we have
\begin{equation*}
    \CE{S_{n+1} - S_n}{\mathcal{G}_n}\leq -\varepsilon,
\end{equation*}
on the event~$\cb{S_n\geq 0}$. The RW~$S_n$ acts like a supermartingale with negative drift until it passes below zero and has controlled jumps
\begin{equation*}
    \CE{\abs{S_{n+1} - S_n}^k}{\mathcal{G}_n}\leq 2A\bb{\alpha r}^k.
\end{equation*}
We aim to use Theorem~$2,\, (iii)$ from~\cite{Lehre:18} to conclude that the distribution of~$\rho$ has an exponentially decaying tail. Using notation of~\cite{Lehre:18} we set~$g(x) = x$, and~$\beta_u(t) =: e^{-\gamma}$ for some~$\gamma>0$ that we fix later. The condition of the cited theorem is
\begin{equation}\label{eq:cited_condition}
    \CE{e^{\lambda\cb{S_{n+1}-S_n}}}{\mathcal{G}_n} - e^{-\gamma} \leq 0,
\end{equation}
on the event~$\cb{S_n \geq 0}$. Using assumption~$(4)$ we observe that for~$\lambda$ sufficiently small
\begin{align*}
    \CE{e^{\lambda\cb{S_{n+1}-S_n}}}{\mathcal{G}_n} &= 1 + \lambda\CE{S_{n+1}-S_n}{\mathcal{G}_n} + \sum_{k\geq 2}  \frac{\lambda^k}{k!}\CE{\bb{S_{n+1}-S_n}^k}{\mathcal{G}_n}\\
    &\leq 1 - \lambda\varepsilon + A\bb{e^{\lambda \alpha r} - \lambda\alpha r - 1}\\
    &= 1-\lambda \varepsilon + o(\lambda^2)\\
    &< 1.
\end{align*}
If we fix~$\gamma$ sufficiently small such that
\begin{equation*}
    \leq 1 - \lambda\varepsilon + A\bb{e^{\lambda \alpha r} - \lambda\alpha r - 1} \leq e^{-\gamma},
\end{equation*}
then conditions of Theorem~$2,\,(iii)$ of~\cite{Lehre:18} are met and
\begin{equation*}
    \bP{\rho > n\mid \mathcal{G}_0} \leq e^{-n\gamma}e^{\lambda S_0}
\end{equation*}
holds for some~$\lambda,\,\gamma >0$. In terms of the random walk~$Z_n$ and stopping times~$\sigma$ and~$\tau$ this implies
\begin{align*}
    \bP{\sigma + n < \tau} &= \bE{\pP\bb{\rho > n\mid \mathcal{G}_0}}\\
    &\leq e^{-n\gamma}\bE{e^{\lambda Z_{\sigma}}}\\
    &\leq e^{-n\gamma}\bE{e^{\lambda X_{\sigma}}}.
\end{align*}
 If the~$\E\Cb{e^{\lambda X_{\sigma}}}$ is finite for some~$\lambda > 0$, then~$\bP{\sigma+n<\tau}$ decays exponentially. To show that this expectation is finite we must first get a good control on the distribution of~$\sigma$. To estimate~$\pP\bb{\sigma\geq k}$ we notice that the position of the left-most site that has~$0$ in their cone of dependence moves with a rate-one Poisson point process. Hence if we let~$N_t$ to be the Poisson point process then for~$k > T$ it holds
 \begin{align*}
     \pP\bb{\sigma > k} &= \pP\bb{\min\cb{t\,:\,N_t = k} < T} \\
     &\leq \pP\bb{N_T\geq k}\\
     &\leq \frac{\bb{eT}^ke^{-T}}{k^k}.
 \end{align*}
 The last inequality is a standard Chernoff bound for the Poisson distribution.
Returning to~$\bE{e^{\lambda X_{\sigma}}}$ we estimate via H\"older
\begin{align}\label{eq:exp_X_sigma_estimate}
    \bE{e^{\lambda X_{\sigma}}}&\leq \sum_{n=0}^{\infty}\bP{\sigma = n}^\frac{1}{2}\bE{e^{2\lambda X_n}}^{\frac{1}{2}}\notag\\
    &\leq \sum_{n=0}^{\infty} \bb{\frac{1}{e n!}}^{\frac{1}{2}}\bE{e^{2\lambda X_n}}^\frac{1}{2}.
\end{align}
To show that this is finite we find an upper bound for~$\bE{e^{\lambda X_n}}$:
\begin{align*}
\bE{e^{\lambda X_n}}& = \bE{e^{2\lambda X_{n-1}}\CE{e^{2\lambda\bb{X_n-X_{n-1}}}}{\mathcal{F}_{n-1}}}\\
&\leq \bE{e^{2\lambda X_{n-1}}\sum_{k=0}^{\infty}\frac{\bb{2\lambda}^k}{k!}\CE{\abs{X_n-X_{n-1}}^k}{\mathcal{F}_{n-1}}}\\
&\leq \bE{e^{2\lambda X_{n-1}}}\frac{A}{1-2\alpha\lambda}\\
&\leq\dots\\
&\leq \bE{e^{2\lambda X_0}}\bb{\frac{A}{1-2\alpha\lambda}}^n.
\end{align*}
In the third step we used the assumption~$(4)$. Setting~$\lambda<\frac{1}{2\alpha}$ and sufficiently small so that~$\bE{e^{2\lambda X_0}}<\infty$ yields that growth of~$\bE{e^{2\lambda X_n}}$ is only exponential in~$n$. Thus the~$\frac{1}{n!}$ term in~(\ref{eq:exp_X_sigma_estimate}) guarantees that the sum converges and~$C:=\bE{e^{\lambda X_{\sigma}}}$ is finite. We have obtained
\begin{equation*}
    \bP{\sigma+n<\tau} \leq Ce^{-n\gamma}
\end{equation*}
for some positive~$\gamma$. This, in combination with the previously attained estimate on tail of~$\sigma$ allows us to conclude that~$\bP{n<\tau}$ decays exponentially:
 \begin{align*}
     \pP\bb{n<\tau}&\leq \pP\bb{\sigma<\frac{1}{2}n,\,n<\tau} + \pP\bb{\sigma\geq \frac{1}{2}n}\\
     &\leq \pP\bb{\sigma + \frac{1}{2}n < \tau} + \pP\bb{\sigma\geq \frac{1}{2}n}\\
     &\leq Ce^{-\frac{1}{2}n\gamma} + e^{-T}\bb{\frac{2eT}{n}}^{\frac{1}{2}n}\\
     &=O\bb{e^{-\frac{1}{2}\gamma n}}.
 \end{align*}
By the estimate in~(\ref{eq:M_estimate}) exponential decay of tail of~$\tau$ implies~$M$ has finite expectation and thus~$t\pP\bb{M>t}\to 0$. The inequality~$\pi_0 \leq M$ finishes the proof. 
\end{proof}

The proof of Theorem~\ref{pro:main_result} is a straightforward application of Lemma~\ref{pro:random_walk_lemma} to random walks tailored to the considered IPS. Assume without loss of generality that~$\mathcal{A} = \cb{0,\dots, \abs{\mathcal{A}}-1}$ and we pick~$a=0$. Let~$\zeta_0$ and~$\xi_0$ be any two initial configurations, whose trajectory we couple in the way described in~(\ref{eq:coupling_construction}). The following observation is crucial for our result:

\begin{itemize}
    \item[(a)] If $ U_n^j < \beta = \min_{\zeta} P_j(0 \mid \zeta) $, then
    \begin{equation*}
    \zeta_t(j) = \xi_t(j) = 0 \quad \text{on } [\tau_n^j, \tau_{n+1}^j) \times \{ j \}.
    \end{equation*}
    
    \item[(b)] If $ \zeta_{\tau_n^{j-}}(j) = \xi_{\tau_n^{j-}}(j) = 0 $ and 
    $U_n^j < 1-\delta = \min_{\zeta_0 : \zeta(j)=0} P_j(1 \mid \zeta)$
    then
    \begin{equation*}
    \zeta_t(j) = \xi_t(j) = 0 \quad \text{on } [\tau_n^j, \tau_{n+1}^j) \times \{ j \}.
    \end{equation*}
\end{itemize}

The above properties mean that at every site, one in $ \frac{1}{\beta} $ updates causes an agreement of zeroes to appear, shared among all trajectories, and this agreement survives for an average of at least $ \frac{1}{\delta} $ updates. This happens independently at every site. We observe that disagreement can't cross these blocks due to nearest-neighbor interaction. Hence, if those agreements of zero cover a sufficiently high proportion of spacetime, then disagreement does not percolate and hence the IPS is ergodic. The following proof makes this heuristic argument rigorous.

\begin{proof}[Proof of Theorem~\ref{pro:main_result}]
    The main idea is to apply Lemma~\ref{pro:random_walk_lemma}. We first mark a subset of spacetime where the coupled processes must agree and then construct two random walks~$X_n, Y_n$ that bound the disagreement. For an illustration we refer to Figure~\ref{fig:illustration_X_n_and_Y_n} below. We first set $ A_j $ to be a subset of $ \mathbb{R}_+ $ made of intervals $ [\tau_n^j, \tau_m^j] $, $ m > n $ such that 
\begin{equation*}
U_{n}^j <\beta, \quad U_{k}^j < 1 - \delta \text{ for } n \leq k < m.
\end{equation*}
The first condition guarantees that an agreement of zeroes appears at~$(\tau_n^j,j)$ and the second that the agreement persists for the next~$m-n$ updates at~$j$.  Then, we have agreement of zeroes on every spacetime block~$ A_j \times \{ j \} $. With that, we set the first random walk
\begin{align*}
X_0 &:= \inf A_0,\\
X_{n+1}:=&\left\{\begin{array}{cc}
    \sup A_{-(n+1)}^c\cap \Cb{0,\,X_n},   & X_n\in A_{-(n+1)}  \\
    \\
     \inf A_{-(n+1)}\cap \Cb{X_n,\infty},   &  X_n\notin A_{-(n+1)}.
   \end{array}  \right.
\end{align*}
The process~$X_n$ at every step checks if the point~$\bb{X_n,\,-(n+1)}$ is one of the points where all trajectories agree to have a zero state. If so then~$X_{n+1}$ traces down the beginning of the interval of agreement in which~$(X_n,\,-(n+1))$ lies. Else~$X_n$ goes up until a new interval of agreement appears. For the other random walk, we set
\begin{align*}
   Y_0&:=0,\\
   Y_{n+1}&:=\left\{\begin{array}{cc}
    \inf A_{-(n+1)}^c\cap \Cb{Y_n,\infty},    & Y_n\in A_{-(n+1)}  \\
    \\
     \min\{\tau_k^{-(n+1)}\geq Y_n\},   &  Y_n\notin A_{-(n+1)}.
   \end{array}  \right.
\end{align*}
The RW~$Y_n$ similarly checks whether the point~$(Y_n,\,-(n+1))$ is in an interval of agreement and if it is then~$Y_{n+1}$ increments up to when this interval ends. Else~$Y_{n+1}$ increments to the next time there is any update at~$(n+1)$, as this is the earliest time a disagreement might spread to site~$-(n+1)$.\\

\begin{figure}[H]
        \centering
        \includegraphics[scale=0.33]{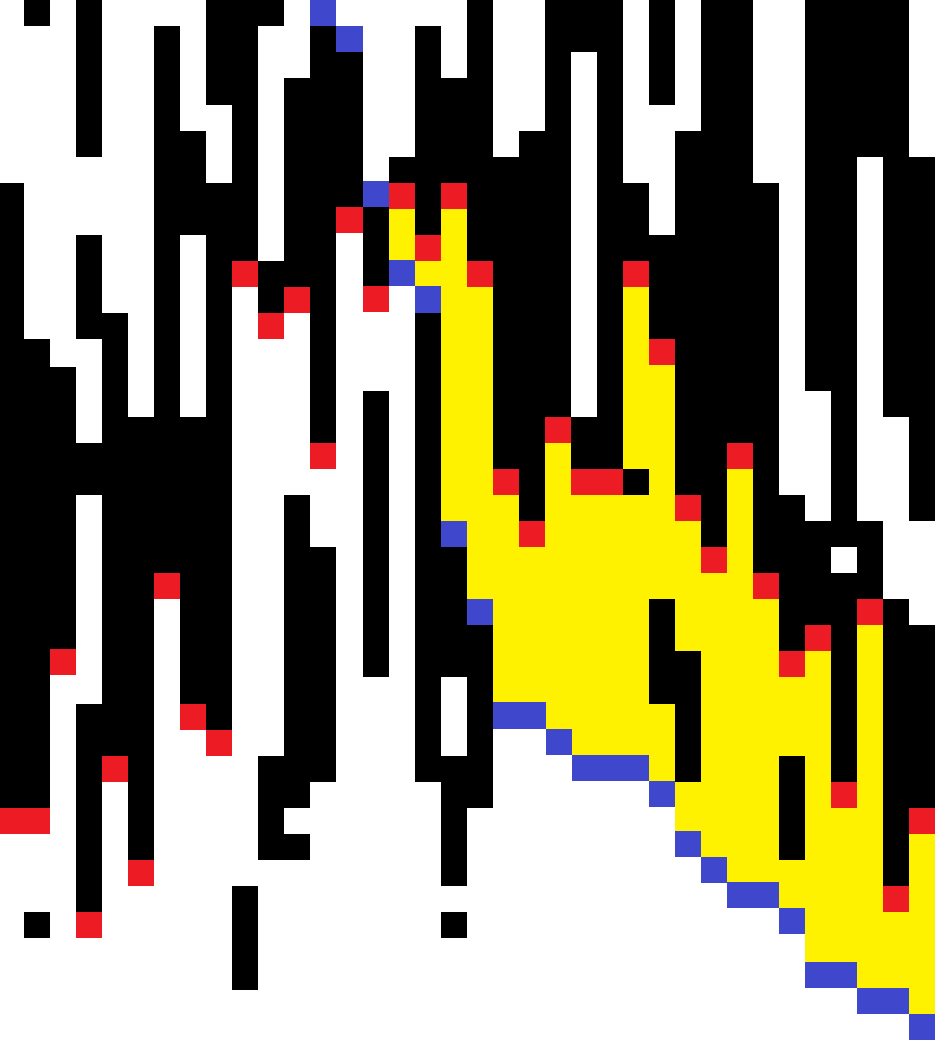}
        \caption{Simulation of the random walks~$X_n$ and~$Y_n$ with~$\beta = \delta = 0.1$. The arrow of time points up, the~$X$-axis represents sites in range~$\cb{-40,\,\dots,\,0}$. The points in spacetime that lie in the intervals of agreement (set~$A$) are colored black. The red marks the points~$\bb{-n,\,X_n}$ and the blue points~$\bb{-n,\,Y_n}$. The region where disagreements may occur is marked in yellow, with the initial disagreement only present at site zero.}\label{fig:illustration_X_n_and_Y_n}
    \end{figure}

If trajectories $ \zeta_t $ and $ \xi_t $ only differ at site zero at time zero, then the set of points in spacetime where they disagree is contained in
\begin{equation*}
\bigcup_{n\in\N}\left(Y_n,\, X_n \right] \times \cb{-n}.
\end{equation*}

The RWs satisfy the Assumptions~$(1),\,(2)$ and~$(5)$ of Lemma~\ref{pro:random_walk_lemma}. The only remaining ingredient is the estimates on the increments of~$X_n$ and~$Y_n$. The verification of Assumption~$(4)$ of the lemma is straightforward. The increments of~$X_n$ and~$Y_n$ are dominated by~Exp$\bb{\min\cb{\beta,\,\delta}}$ random variables, so we have
\begin{align*}
    \E\Cb{\abs{X_{n+1}-X_n}^k}&\leq \frac{k!}{\min\cb{\beta,\,\delta}^k},\\
    \E\Cb{\abs{Y_{n+1}-Y_n}^k}&\leq \frac{k!}{\min\cb{\beta,\,\delta}^k}.
\end{align*}

Now we verify the Assumption~$\bb{3}$, with step size~$r=1$. We compute the expected increments of $Y_n$ first. To do that we need to find $\pP\bb{t\in A_j}$. We first compute its derivative 
\begin{align*}
    \frac{\partial}{\partial t} \mathbb{P}(t \in A_j) &= -\delta \mathbb{P}(t \in A_j) + \beta \mathbb{P}(t \notin A_j)\\
    &=-(\beta + \delta) \mathbb{P}(t \in A_j) + \beta.
\end{align*}
Since~$\pP\bb{0\in A_j} = 0$ then
\begin{equation*}
    \mathbb{P}(t \in A_j) = \frac{\beta}{\beta + \delta} - \frac{\beta}{\beta + \delta} e^{-(\beta + \delta)t}.
\end{equation*}
By the memoryless property of the exponential distribution, we have
\begin{align}\label{eq:drift_Y}
    \CE{Y_{n+1}-Y_n}{Y_n} &= \frac{1}{\delta}\pP\bb{t\in A_{-(n+1)}}|_{t=Y_n} + \pP\bb{t\notin A_{-(n+1)}}|_{t=Y_n}\notag\\
    & = \frac{\beta+\delta^2}{\delta\bb{\beta+\delta}} + O\bb{e^{-\bb{\beta + \delta}Y_n}}.
\end{align}
To calculate $\CE{X_{n+1} - X_n}{X_n} $, we consider two cases: either~$X_n\notin A_{-(n+1)}$ or~$X_n\in A_{-(n+1)}$. In the first case~$X_{n+1}-X_n$ has the~Exp$(\beta)$ distribution and we have
\begin{equation}\label{eq:drift_X_up}
    \CE{\bb{X_{n+1}-X_n}\ind\bb{X_n\notin A_{-(n+1)}}}{X_n} = \frac{\delta}{\beta\bb{\beta+\delta}} + O\bb{e^{-\bb{\beta+\delta}X_n}}.
\end{equation}
To compute the contribution to the drift of~$X_n$ from the jumps down we introduce the random variable that for each point in spacetime~$(t,j)$ marks the last time before~$t$ that a zero-killing update occured at site~$j$, if there is any:
\begin{equation*}
    \tau_{\delta}(t,j):=\sup\cb{\tau_n^j \leq t : U_n^j\geq 1-\delta}\cup\cb{0}.
\end{equation*}
If a zero-creating~$(U_n^j<\beta)$ update occurs at~$j$ in the interval~$\Cb{\tau_{\delta}(t,j),t}$ then~$t\in A_j$. Let~$\tau_{\beta}(t,j)$ be the time of first zero-creating update after~$\tau_\delta(t,j)$, i.e.
\begin{equation*}
    \tau_{\beta}(t,j):=\inf\cb{\tau_n^j > \tau_{\delta}(t,j) : U_n^j <\beta}.
\end{equation*}
Then we can express the contribution of jumps down
\begin{equation*}
    \CE{\bb{X_{n+1}-X_n}\ind\bb{X_n\in A_{-(n+1)}}}{X_n}
\end{equation*}
    using the new random variables as
\begin{equation*}
    \CE{\min\cb{0,\tau_{\beta}\bb{X_n,-(n+1)}-X_n}}{X_n}.
\end{equation*}
To calculate this expectation we only need to know the distribution of
\begin{equation*}
    \tau_{\beta}\bb{t,j}-t,
\end{equation*}
for arbitrary~$t$. The distribution does not depend on~$j$.
We can represent this random variable as a sum of two random variables with known distributions, which we argue are independent:
\begin{equation*}
    \tau_{\beta}\bb{t,j}-t = \Cb{\tau_{\beta}\bb{t,j} - \tau_{\delta}\bb{t,j}} + \Cb{\tau_{\delta}\bb{t,j}-t}.
\end{equation*}
The zero-killing updates arrive with rate~$\delta$, so the distribution of~$t-\tau_\delta(t,j)$ is that of~$\min\cb{t,W}$, where~$W$ is an Exp$\bb{\delta}$ random variable. Because the zero-killing and zero-creating updates arrive independently the increment~$\tau_{\beta}(t,j) - \tau_{\delta(t,j)}$ has distribution~Exp$(\beta)$ and is independent of~$\tau_{\delta(t,j)}$.
The contribution to drift of~$X_n$ from jumps down is thus
\begin{equation}\label{eq:drift_X_down}
\int_{\R_+^2}\beta\delta e^{-\beta x -\delta y}\min\cb{0,x-\min\cb{X_n,y}}d(x,y) = -\frac{\beta}{\delta(\beta+\delta)} + O\bb{e^{-\delta X_n}}.
\end{equation}
By the combination of~(\ref{eq:drift_Y}),~(\ref{eq:drift_X_up}) and~(\ref{eq:drift_X_down}) the RW~$Z_n := X_n - Y_n$ satisfies condition~$(3)$ of Lemma~\ref{pro:random_walk_lemma} if
\begin{equation}\label{eq:drift_before_time_scaling}
    -\frac{\beta+\delta^2}{\delta\bb{\beta+\delta}}+\frac{\delta}{\beta\bb{\beta+\delta}}-\frac{\beta}{\delta\bb{\beta+\delta}}<0.
\end{equation}
To finish off the proof, we apply the Time-Scaling Lemma. When the IPS is slowed down, the coefficients~$\beta$ and~$\delta$ converge to zero, but the ratio~$\frac{\beta}{\delta}$ remains invariant. Thus~(\ref{eq:drift_before_time_scaling}) simplifies to
\begin{equation*}
    -2\frac{\beta}{\delta} + \frac{\delta}{\beta} < 0,
\end{equation*}
which is equivalent to~$\delta < \sqrt{2}\beta$.
\end{proof}

\section*{Acknowledgment}
This material is based upon work supported by the National Science Foundation under Grant No. DMS-2348113. The authors also want to thank Marek Biskup for insightful discussions on the topic.

\bibliographystyle{plain}
\bibliography{bibliography}

\end{document}